\documentclass[letterpaper, 10 pt, conference]{ieeeconf}
\IEEEoverridecommandlockouts       
\usepackage{color}
\usepackage[utf8]{inputenc}
\usepackage{amsthm}
\usepackage{amsmath}
\usepackage{amssymb}
\usepackage{graphicx}
\usepackage{hyperref}
\DeclareMathOperator{\atantwo}{atan2}

\newtheorem{theorem}{Theorem}

\newtheorem{assumption}{Assumption}

\newtheorem{definition}{Definition}
\title{Navigation in Unknown Environments Using Safety Velocity Cones}
\author{Soulaimane Berkane,~\IEEEmembership{Member,~IEEE.}\thanks{This research work is supported by the National Research Council of Canada, grant $n^{o}$: NSERC-DG RGPIN-2020-04759. The author is with the Department of Computer Science and Engineering, University   of   Quebec   in   Outaouais, 101 St-Jean Bosco, Gatineau, QC, J8X 3X7, Canada.  Email: \href{mailto:soulaimane.berkane@uqo.ca}{soulaimane.berkane@uqo.ca}.}}
\date{December 2019}

\begin{document}

\maketitle
\begin{abstract}
    We rely on Nagumo's invariance theorem to develop a new approach for navigation in unknown environments of arbitrary dimension. The idea consists in projecting the nominal velocities (that would drive the robot to the target in the absence of obstacles) onto Bouligand's tangent cones (referred to as the {\it safety velocity cones}) when the robot is close to the boundary of the free space. The proposed projection-based controller is explicitly constructed to guarantee safety and convergence to a set of Lebesgue measure zero that contains the target. For specific free spaces such as Euclidean sphere worlds, the convergence to the target is guaranteed from almost all initial conditions in the free space. We provide a version of the controller (generating a continuous and piecewise smooth closed-loop vector field) relying on the robot's current position and local range measurements (e.g., from LiDAR or stereo vision) without global prior knowledge about the environment.
\end{abstract}
\section{Introduction}
Obstacle avoidance is a long-lasting problem with a large body of research work in the robotics and control communities. Existing approaches to the obstacle avoidance problem can be classified into two main categories: global (map-based) methods that require the global knowledge of the environment and local (reactive) methods that use only local knowledge of the obstacles in the vicinity of the robot. Among the global methods, artificial potential fields-based approaches \cite{Khatib1986Real-timeRobots,Koditschek1990RobotBoundary} are computationally efficient and come usually with rigorous proofs for safety and convergence. However, the design of such potential fields is non-trivial for generic environments due to the complications that arise from the presence of local minima. 
The navigation functions of Koditscheck and Rimon \cite{Koditschek1990RobotBoundary} are carefully constructed, for Euclidean sphere worlds \cite[p. 414]{Koditschek1990RobotBoundary}, to get an artificial potential field with the nice property that all but one of the critical points are saddles with the remaining critical point being the desired reference. Since then, the navigation function-based approach has been extended in many different directions; e.g., for multi-agent systems \cite{Tanner2005FormationFunctions,Dimarogonas2006TotallyFunctions,Roussos2013DecentralizedRobots}, and for focally admissible obstacles \cite{Filippidis2013NavigationSurfaces}. The major drawback of navigation functions is that they require an unknown tuning of parameters to eliminate local minima. Other recent global approaches that do not require any parameter tuning to eliminate local minima are the navigation transform \cite{Loizou2017TheTransformation}, the prescribed performance control \cite{Vrohidis2018PrescribedNavigation}, and the hybrid control approach \cite{Berkane2019} which eliminates also the saddle points.

On the other hand, having a reactive solution to safely navigate in unknown environments is more desirable in practical autonomous robots applications. One of the simplest reactive motion planning algorithms are the family of Bug algorithms \cite{Lumelsky1986DynamicEnvironment,Lumelsky1990IncorporatingFunction} used to navigate in planar environments.  Locally computable navigation functions have been proposed in  \cite{Lionis2007LocallyWorlds} to navigate in unknown sphere worlds and adjustable navigation functions are proposed to gradually update the tuning parameter upon the discovery of new obstacles \cite{Filippidis2011AdjustableWorlds}. On the other hand, {\it purely reactive} algorithms use only current sensor data to generate the safe trajectories and control law. Potential fields can be generated in real-time using current sensor readings. However, purely reactive gradient-following potential field approaches always run the risk of getting stuck in local minima \cite{Choset2005PrinciplesImplementations}. Recently, a new reactive control algorithm has been developed to navigate in environments cluttered with unknown but sufficiently separated and strongly convex obstacles \cite{Arslan2019Sensor-basedWorlds}. Moreover, for planar navigation, the algorithm is implementable using practical sensing models. 
\subsection{Contributions of the Paper}
In this paper we address the problem of robot navigation in an arbitrary unknown environment by constructing a vector field that simultaneously solves the motion planning and control problems. The free space of the robot is assumed to be a closed subset of the $n-$dimensional Euclidean space. We then exploit Nagumo's theorem for invariant sets \cite{Blanchini2008Set-theoreticControl} to guarantee safety of the robot in this environment by projecting the nominal vector field (nominal velocity of the robot) onto the (Bouligand's) tangent cones \cite{Bouligand1932IntroductionDirecte} whenever the robot is close to the boundary. When the boundary of the free space is a smooth submanifold of codimension equals one, the tangent cones are {\it flat} (thus convex) and continuously varying along the boundary.  The resulting vector field is piecewise continuously differentiable and the unique Fillipov solution is shown to converge  to a set of Lebegue measure zero containing the desired exponentially stable location (target) as well as potentially other undesired equilibrium configurations. For free spaces that are modelled as {\it sphere worlds} \cite{Koditschek1990RobotBoundary}, we show that the resulting vector field has the same capabilities as navigation functions without requiring the assumption of global prior knowledge. The undesired equilibrium configurations are explicitly characterized and shown to be unstable (hence almost global asymptotic stability of the target). 

For practical implementation, we show that the proposed algorithm is intrinsically sensor-based in the sense that the vector field can be generated on the fly in real-time using onboard sensors such as LiDARs, stereo cameras,etc; depending on the application at hand. This is achieved by characterizing our controller using the distance function to the boundary of the free space. Moreover, we propose a smoothed version of the control law that removes the discontinuity of the vector field while adding a safety maneuvering margin without compromising the obtained navigation capabilities. 
\subsection{Organization of the Paper and Notation}
This paper is organized as follows. In Section \ref{section:nagumo} we recall the well known Nagumo's theorem for invariant sets, which is stated using the concept of tangent cones. Section \ref{section:main} presents the main theoretical findings of this paper, which are based on a newly proposed projection-based control strategy for safe navigation. In Section \ref{section:distance} we characterize our control law using the distance to the obstacles function which allows also to smooth-out the discontinuouty of the controller. In Section \ref{section:practical}, we discuss some of the practical considerations related to the sensor-based implementation of the proposed navigation controller. Simulation results in a 2D environment are provided in Section \ref{section:simulation} to demonstrate the effectiveness of the proposed navigation algorithm. Finally, we provide in Section \ref{section:conclusion} some concluding remarks and related future work.

{\bf Notation:}
$\mathbb{R}$ and $\mathbb{R}_{<0}$ denote, respectively, the set of reals and negative reals. $\mathbb{R}^n$ is the $n$-dimensional Euclidean space and $\mathbb{S}^{n-1}$ is the $(n-1)$-dimensional unit sphere embedded in $\mathbb{R}^n$. The topological interior (resp. boundary) of a subset $\mathcal{A}\subset\mathbb{R}^n$ is denoted by $\mathbf{int}(\mathcal{A})$ (resp. $\partial\mathcal{A}$). The complement of $\mathcal{A}$ in $\mathbb{R}^n$ is denoted by $\complement\mathcal{A}$ and its closure by $\overline{\mathcal{A}}$. The Euclidean norm of $x\in\mathbb{R}^n$ is defined as $\|x\|:=\sqrt{x^\top x}$ where $(\cdot)^\top$ stands for the transpose operator. We denote by $\mathcal{B}(x,r):=\{y\in\mathbb{R}^n: \|x-y\|<r\}$ the open ball of radius $r$ that is centered at $x$. Given a non-empty subset $\mathcal{A}\subset\mathbb{R}^n$, the distance function from a point $x\in\mathbb{R}^n$ to $\mathcal{A}$ is defined as $\mathbf{d}_{\mathcal{A}}(x):=\mathrm{inf}_{y\in\mathcal{A}}\|y-x\|$ and the projection of $x$ onto $\overline{\mathcal{A}}$ is given by the set-valued map  $\mathbf{P}_{\mathcal{A}}(x):=\{y\in\overline{\mathcal{A}}: \|y-x\|=\mathbf{d}_{\mathcal{A}}(x)\}$. 

\section{Nagumo’s theorem for invariant sets}\label{section:nagumo}
Nagumo’s theorem is one of the main important tools used in the characterization of positively invariant sets for continuous-time dynamical systems. There are different ways to state the theorem but the most intuitive and easy way to do so is to use the notion of {\it tangent cones}. There are different definitions of tangent cones (giving rise to different cones) such as Clarke \cite{Clarke1990OptimizationAnalysis}, Bony \cite{Bony1969PrincipeDegeneres}, and Bouligand \cite{Bouligand1932IntroductionDirecte}. These cones coincide for a convex set, but they can differ on more general sets. Here we recall the definition of (Bouligand's) tangent cones which will be used to state Nagumo's theorem.
\begin{definition}[Bouligand's tangent cone \cite{Bouligand1932IntroductionDirecte,Blanchini2008Set-theoreticControl}]
Given a closed set $\mathcal{X}\subseteq\mathbb{R}^n$, the tangent cone to $\mathcal{X}$ at $x$ is the subset of $\mathbb{R}^n$ defined by
 \begin{align}
      \mathbf{T}_{\mathcal{X}}(x):=\left\{z: \liminf_{\tau\to 0^+}\frac{\mathbf{d}_{\mathcal{X}}(x+\tau z)}{\tau}=0\right\}.
 \end{align}
\end{definition}%
Roughly speaking, the tangent cone for $x\in\mathcal{X}$ is the set that contains all the vectors whose directions point from $x$ inside or tangent to the set $\mathcal{X}$. It should be noted that the tangent cone $\mathbf{T}_{\mathcal{X}}(x)$ is closed and non-trivial only at the boundary of $\mathcal{X}$. In fact, for all $x\in\mathbf{int}(\mathcal{X})$ (points in the interior of $\mathcal{X}$), we have $\mathbf{T}_{\mathcal{X}}(x)=\mathbb{R}^n$, and for all $x\notin\mathcal{X}$ we have $\mathbf{T}_{\mathcal{X}}(x)=\emptyset$. Now we recall Nagumo’s invariance theorem.
\begin{theorem}[Nagumo 1942 \cite{Nagumo1942UberDifferentialgleichungen,Blanchini2008Set-theoreticControl}]\label{theorem:nagumo}
Consider the system $\dot x(t)=f(x(t))$ and assume that for each initial condition $x(0)$ in an open set $\mathcal{O}$ it admits a unique solution defined for all $t\geq 0$. Then the closed set $\mathcal{X}\subset\mathcal{O}$ is forward (positively) invariant if and only if
\begin{equation}\label{eq:nagumo}
    f(x)\in\mathbf{T}_{\mathcal{X}}(x),\quad\forall x\in\mathcal{X}.
\end{equation}
\end{theorem}

Nagumo's theorem has a very intuitive and simple geometric interpretation. In fact, condition \eqref{eq:nagumo} says that the vector field $f(x)$ (velocity of $x$) must point inside (or it is tangent to) the set $\mathcal{X}$ at each position $x$. As discussed above, since $\mathbf{T}_{\mathcal{X}}(x)$ is non-trivial only at the boundary points $x\in\partial\mathcal{X}$, Nagumo's condition \eqref{eq:nagumo} is checked only at these boundary points since it is trivially satisfied for all $x\in\mathbf{int}(\mathcal{X})$.
\begin{figure}
    \centering
    \includegraphics[width=0.9\columnwidth]{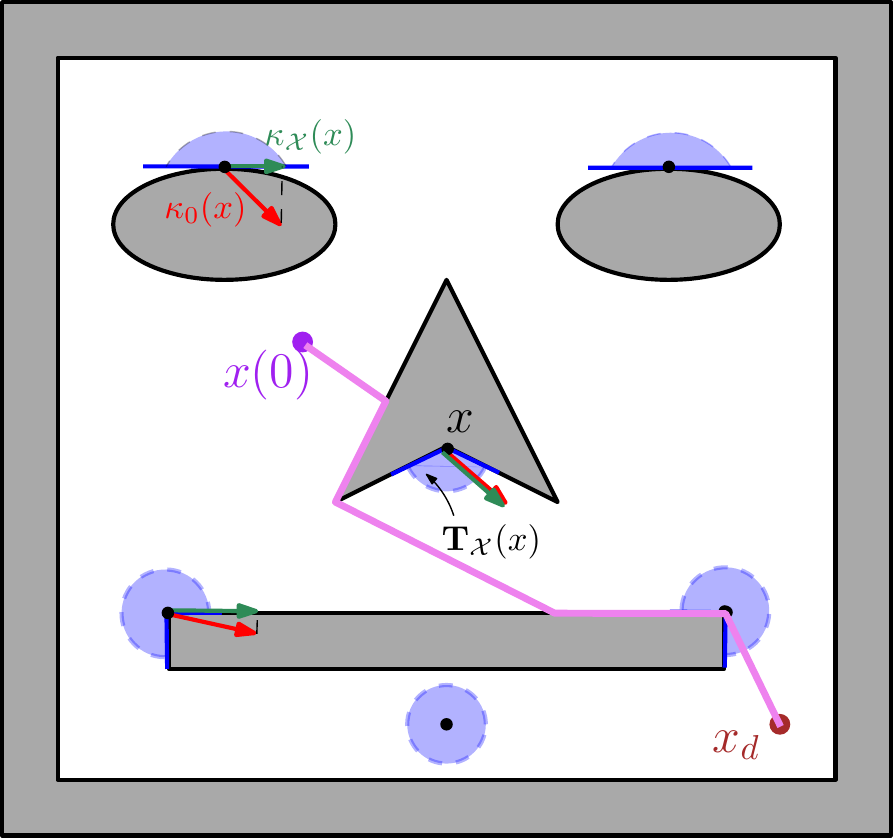}
    \caption{A face-shaped free space $\mathcal{X}$ (interior in white, boundary in solid black). At each position $x$, the nominal controller $\kappa_0(x)$ (in bold red) is projected onto the safety velocity cone (in blue) to give rise to the safe control $\kappa_\mathcal{X}(x)$ of \eqref{eq:dx-set} (in bold green). The resulting piecewise smooth vector field (in bold pink) is an orchestration of two modes: stabilization and avoidance.}
    \label{fig:nagumo}
\end{figure}

\section{Projection onto Tangent cones for safe navigation}\label{section:main}
In this section, we propose a simple and very efficient safe navigation approach that is based on Nagumo's theorem and tangent cones. The proposed approach is intrinsically local (does not require prior global knowledge of the environment) and considers a very generic class of free spaces.
\subsection{Problem Formulation}
We consider a point-mass robot moving in a closed (not necessary bounded) subset $\mathcal{X}$ of the $n$-dimensional Euclidean space $\mathbb{R}^n$. The set $\mathcal{X}$ is often called the obstacle-free set or the free space. The robot is moving according to the following first-order dynamics (fully actuated single integrator):
\begin{align}\label{eq:dx1}
    \dot x=u,
\end{align}
where $x\in\mathbb{R}^n$ is the position of the robot, initially at $x(0)\in\mathcal{X}$, and $u\in\mathbb{R}^n$ is the velocity (control input). The navigation problem consists in finding a suitable control policy (vector field) $\kappa_\mathcal{X}:\mathcal{X}\to\mathbb{R}^n$ such that, under the control law $u=\kappa_\mathcal{X}(x)$, the robot's position is asymptotically stabilized at some given desired location $x_d\in\mathbf{int}(\mathcal{X})$ while staying in the set $\mathcal{X}$ for all times (forward invariance of $\mathcal{X}$). The resulting closed-loop system is written as
\begin{align}\label{eq:dx2}
    \dot x=\kappa_\mathcal{X}(x).
\end{align}
Moreover, the control policy $\kappa_\mathcal{X}$ must be locally computed in the sense that only the local information about the free space $\mathcal{X}$ is used when generating in real-time the resulting vector field. 
\subsection{Invariance Through Safety Velocity Cones}
Here we exploit the fact that the notion of safety is tightly related to the notion of invariance. In view of \eqref{eq:dx2} and by direct application of Nagumo's theorem (Theorem \ref{theorem:nagumo}), if we want to ensure the safety specification, the control policy $\kappa_\mathcal{X}(\cdot)$ must satisfy the  condition
\begin{equation}\label{eq:kappa:constraint}
        \kappa_\mathcal{X}(x)\in\mathbf{T}_{\mathcal{X}}(x),\quad\forall x\in\mathcal{X}.
\end{equation}
Therefore, for all $x\in\mathcal{X}$, it is reasonable to choose a control law that solves the following constrained optimization problem
\begin{equation}\label{eq:optimization}
        \min_u \|u-\kappa_0(x)\|^2\quad\textrm{subject to\;}u\in\mathbf{T}_{\mathcal{X}}(x), \forall x\in\mathcal{X},
\end{equation}
where $\kappa_0(x)$ is a nominal control law that ensures  the stabilization task (motion to goal) in the absence of obstacles. In other words, at each point in $\mathcal{X}$, the control law that we would like to apply will be equal to the control that is the {\it closest} one  to the nominal control law under the constraint \eqref{eq:kappa:constraint}. This problem is referred to as the {\it nearest point problem} and the solution to this minimization is given by the {\it metric projection} of $\kappa_0(x)$ onto the set $\mathbf{T}_{\mathcal{X}}(x)$, i.e., 
$
    \mathbf{P}_{\mathbf{T}_{\mathcal{X}}(x)}(\kappa_0(x)).
$
Since $\mathbf{P}_{\mathbf{T}_{\mathcal{X}}(x)}(\cdot)$ is set-valued, the resulting closed-loop system can be written as a differential inclusion
\begin{align}\label{eq:dx-set}
    \dot x=\kappa_\mathcal{X}(x)\in\mathbf{P}_{\mathbf{T}_{\mathcal{X}}(x)}(\kappa_0(x)),
\end{align}
which has (by construction) solutions that are necessarily safe (will stay in $\mathcal{X}$). Since velocities are projected onto the tangent cone $\mathbf{T}_{\mathcal{X}}(x)$ to guarantee safety, we term these tangent cones as {\it safety velocity cones} (SVCs) when used for this navigation problem. A striking feature of this controller is that the SVCs are characterized locally near each point of the boundary and, hence, the proposed navigation scheme is intrinsically a local control strategy that does not require global knowledge of the free space\footnote{In Section \ref{section:practical}, we will address the practical implementation of this control strategy using on-board sensing, which results in a reactive sensor-based navigation scheme.}.  Note that, since  $\mathbf{T}_{\mathcal{X}}(x)=\mathbb{R}^n$ for points in the interior of $\mathcal{X}$, one can writes
\begin{align}\label{eq:dx-set}
    \dot x=\kappa_\mathcal{X}(x)\in\begin{cases}
    \{\kappa_0(x)\},&x\in\mathbf{int}(\mathcal{X})\\
    \mathbf{P}_{\mathbf{T}_{\mathcal{X}}(x)}(\kappa_0(x)),&x\in\partial\mathcal{X},
    \end{cases}
\end{align}
which shows the discontinuous\footnote{See Section \ref{section:distance} for the smooth version of this control law.} nature of the resulting vector field, see Fig.~\ref{fig:nagumo} . 

Note also that, in view of the famous Hilbert project theorem in convex analysis, if the tangent cone $\mathbf{T}_{\mathcal{X}}(x)$ is convex at each $x$, the projection map $\mathbf{P}_{\mathbf{T}_{\mathcal{X}}(x)}(\cdot)$ is single-valued. In this case the differential inclusion \eqref{eq:dx-set} reduces to a differential equation with discontinuous right-hand side.  For example, the SVCs of the face-shaped set in Fig.~\ref{fig:nagumo} are not all convex since at the corners of the mouth (bottom rectangle) the SVCs are not convex. Note that  convexity of the tangent cone does not necessary require convexity nor smoothness of the set $\mathcal{X}$. For instance, if we remove the mouth from the set in Fig.~\ref{fig:nagumo}, the resulting free space has SVCs that are convex everywhere; although the set has a non-smooth boundary and is obviously non-convex. Moreover, arbitrary closed sets with smooth boundaries have tangent cones that are half-spaces (thus convex) and these will be treated in the next subsection. 

Finally, it is worth pointing out that, in the 2D context, the obtained navigation trajectories are similar to those obtained using the traditional Bug planning algorithm \cite{Lumelsky1986DynamicEnvironment} where the planner switches between motion-to-goal and boundary-following maneuvers. Our proposed obstacle avoidance strategy extends this to an Euclidean space with arbitrary dimension while also uniting the planning and feedback stabilization tasks, in the same spirit of potential field methods \cite{Koditschek1990RobotBoundary,Khatib1986Real-timeRobots}, and providing theoretical guarantees.  

\subsection{Free Spaces with Smooth Boundaries}\label{section:main-smooth}
Here we focus our attention to free spaces with smooth boundaries. This is motivated by the distance-based characterization wich will be presented in Section \ref{section:distance} for positive reach sets, which intrinsically generate practical free spaces that have smooth boundaries.  We assume, hereafter, that $\partial\mathcal{X}$ is a hypersurface of $\mathbb{R}^n$ that is {\it orientable\footnote{An orientable hypersurface is an $(n-1)-$dimensional submanifold of $\mathbb{R}^n$ that admist a global unit normal vector field defined along it \cite[Definition 4.14]{Lee2009ManifoldsGeometry}.}, without boundary, and twice continuously differentiable}. Therefore, there exists a continuously differentiable map (called also the Gauss map)  $\nu:\partial\mathcal{X}\to\mathbb{S}^{n-1}$ that associates to each $x\in\partial\mathcal{X}$ the (outward) normal unit vector $\nu(x)$ to the hypersurface $\partial\mathcal{X}$ at $x$. The tangent cone at any $x\in\partial\mathcal{X}$ is, therefore, given by the half-space
\begin{align}\label{eq:T:smooth}
    \mathbf{T}_{\mathcal{X}}(x)=\big\{z\in\mathbb{R}^n: \nu(x)^\top z\leq 0\big\},\quad \forall x\in\partial\mathcal{X}.
\end{align}
Since $\mathbf{T}_{\mathcal{X}}(x)$ is a half-space it is convex and thus the solution of the nearest point problem in \eqref{eq:optimization} is unique and can be explicitly solved. Let us derive the explicit expression of this projection map. If $x\in\partial\mathcal{X}$ and $\nu(x)^\top\kappa_0(x)\leq 0$ then $\kappa_0(x)\in\mathbf{T}_{\mathcal{X}}(x)$ and hence we have $\kappa_0(x)$ is a solution to \eqref{eq:optimization}. Otherwise when $x\in\partial\mathcal{X}$ and $\nu(x)^\top\kappa_0(x)>0$, the closest point is obtained via the orthogonal projection onto the hyperplane $\nu(x)^\top z=0$, which is given by \cite{Meyer2000MatrixAlgebra}
\begin{align}
    \Pi(\nu(x))\kappa_0(x):=(I_n-\nu(x)\nu(x)^\top)\kappa_0(x).
\end{align}
Therefore, the projection of $\kappa_0(x)$ onto $\mathbf{T}_{\mathcal{X}}(x)$ is given by
\begin{align}\label{eq:kappa:projection}
     \kappa_\mathcal{X}(x)&=\nonumber
    \mathbf{P}_{\mathbf{T}_{\mathcal{X}}(x)}(\kappa_0(x))\\
    &=\begin{cases}
    \kappa_0(x),&x\in\mathbf{int}(\mathcal{X})\;\mathrm{or}\;\\
    &\nu(x)^\top\kappa_0(x)\leq 0,\\
    \Pi(\nu(x))\kappa_0(x),&x\in\partial\mathcal{X}\;\mathrm{and}\;\\
    &\nu(x)^\top\kappa_0(x)\geq 0,
    \end{cases}
\end{align}
which explicitly defines our control policy for free spaces with smooth boundary. Roughly speaking, the resulting (discontinuous) controller applies the nominal control as long as the robot is not at the boundary of the free space (stabilization mode). When the robot is at the boundary of the free space, the nominal controller is projected onto the tangent cone to ensure safety (avoidance mode). Note that during the avoidance mode and if the nominal controller is pointing inside the set $\mathcal{X}$, the projection does not alter the nominal controller (i.e., when $\nu(x)^\top\kappa_0(x)\leq 0$). This is a very intuitive and simple idea that we have formulated rigorously using the projection onto tangent cones. Of course from a practical point of view one needs to consider a safety margin from the boundary of the free space and this will be discussed later in Section \ref{section:distance}. Now we are ready to state our first result.
\begin{theorem}[Safety]\label{theorem:safety}
Consider a free space that is described by a closed set $\mathcal{X}\subset\mathbb{R}^n$ such that $\partial\mathcal{X}$ is an orientable and $C^2-$hypersurface. Consider the kinematic system \eqref{eq:dx1} under the control law \eqref{eq:kappa:projection} where $\kappa_0(\cdot)$ is continuously differentiable. Then, the closed-loop system admits a unique solution (in the sense of Fillipov) for all $t\geq 0$ and the set $\mathcal{X}$ is forward invariant.
\end{theorem}

\begin{proof}
First we show that the vector field $\kappa_\mathcal{X}(x)$ is piecewise continuous. In fact, $\kappa_\mathcal{X}(x)$ is continuous on $\mathbf{int}(\mathcal{X})$ since  $\kappa_0(x)$ is continuously differentiable. Moreover, for $x\in\partial\mathcal{X}$, we have
\begin{align}
    \kappa_\mathcal{X}(x)
    =\begin{cases}
    \kappa_0(x),&\nu(x)^\top\kappa_0(x)\leq 0,\\
    \Pi(\nu(x))\kappa_0(x),&\nu(x)^\top\kappa_0(x)\geq0,
    \end{cases}
\end{align}
which is continuous since $\nu(x)$ is continuously differentiable and, at $\nu(x)^\top\kappa_0(x)=0$,  it holds that $\Pi(\nu(x))\kappa_0(x)=\kappa_0(x)$. Therefore, $\kappa_\mathcal{X}(x)$ is discontinuous only at the boundary points. Moreover, for each $x\in\partial\mathcal{X}$, the vector field $\kappa_\mathcal{X}(x)$ points naturally into $\mathbf{int}(\mathcal{X})$ since it projects onto the tangent cone when at the boundary points. It follows from \cite[Proposition 5]{Cortes2008DiscontinuousSystems} that there exists a unique Filippov solution starting from any initial condition. Finally since $\kappa_\mathcal{X}(x)\in\mathbf{T}_{\mathcal{X}}(x)$ for all $x\in\mathcal{X}$, it follows from Nagumo's theorem that the set $\mathcal{X}$ is forward invariant. 
\end{proof}
Theorem \ref{theorem:safety} shows that safety is guaranteed regardless of the nominal controller $\kappa_0(x)$ and under a mild condition on free-space $\mathcal{X}$ (smooth boundary). This is an important result since safety is often a hard constraint and should be satisfied regardless of the free space or the control strategy applied. On the other hand, convergence to the target is an objective that it tightly related to the topology of the free space and the chosen $\kappa_0(x)$. Nevertheless, in view of the fact that our controller equals the nominal controller ($u=\kappa_0(x)$) for points in the interior of $\mathcal{X}$ where the target point $x_d$ lies in, there exists a ball around $x_d$ from which convergence to $x_d$ is ensured. However, the region of attraction of the equilibrium $x=x_d$ can be smaller than $\mathcal{X}$, i.e. there exists initial conditions in $\mathcal{X}$ from which the robot will not reach $x_d$. To further characterize the equilibria set of the closed-loop system, we pick for illustration the state feedback 
\begin{align}\label{eq:kappa0}
        \kappa_0(x)=-k(x-x_d),\quad k>0.
\end{align}
Next we state our second result which is related to the convergence task (motion to goal).
\begin{theorem}[Motion to Goal]\label{theorem:convergence}
Consider a free space that is described by a closed set $\mathcal{X}\subset\mathbb{R}^n$ such that $\partial\mathcal{X}$ is an orientable and $C^2-$hypersurface. Consider the kinematic system \eqref{eq:dx1} under the control law \eqref{eq:kappa:projection}  with $\kappa_0(\cdot)$ as in \eqref{eq:kappa0}. Then, the distance $\|x-x_d\|$ is non-increasing, the equilibrium $x=x_d$ is exponentially stable, and the unique Filippov solution converges to the set $\{x_d\}\cup\mathcal{E}$, where $\mathcal{E}:=\{x\in\partial\mathcal{X}: x=x_d+\lambda\nu(x), \lambda\in\mathbb{R}_{<0}\}$.
\end{theorem}

\begin{proof}
For simplicity let us denote $f(x):=\mathbf{P}_{\mathbf{T}_{\mathcal{X}}(x)}(\kappa_0(x))$. First we construct the differential inclusion that captures the unique Fillipov's solution as follows:
\begin{align}\label{eq:inclusion}
    x(t)\in\mathbf{F}[f](x(t)),
\end{align}
where $\mathbf{F}[f]:\mathbb{R}^n	\rightrightarrows\mathbb{R}^n$ is a set-valued mapping defined by
\begin{align}
    \mathbf{F}[f](x):=\bigcap_{\delta>0}\bigcap_{\mu(S)=0}\overline{\mathrm{co}}(f(\mathcal{B}(x,\delta))\setminus S),
\end{align}
where $\overline{\mathrm{co}}$ denotes the convex closure hull and $\mu$ is the Lebesgue measure. Since $f$ is piecewise continuous and discontinuous only at the boundary points, we can write
\begin{align}\label{eq:F}
    \mathbf{F}[f](x)=\begin{cases}
    \{\kappa_0(x)\},\\
    \qquad x\in\mathbf{int}(\mathcal{X})\;\mathrm{or}\;\nu(x)^\top\kappa_0(x)\leq 0,\\
    \{(\alpha I_n+(1-\alpha)\Pi(\nu(x)))\kappa_0(x),\alpha\in[0,1]\},\\
    \qquad x\in\partial\mathcal{X}\;\mathrm{and}\;\nu(x)^\top\kappa_0(x)\geq 0.
    \end{cases}
\end{align}
In other words, when $x\in\partial\mathcal{X}$ and $\nu(x)^\top\kappa_0(x)\geq 0$, $\mathbf{F}[f](x)$ is the set of all possible convex combinations of $\Pi(\nu(x))\kappa_0(x)$ and $\kappa_0(x)$ while in all other cases it is a singleton. Consider the following continuously differentiable positive definite function
\begin{align}
    \mathbf{V}(x)=\frac{1}{2}\|x-x_d\|^2.
\end{align}
The set-valued Lie derivative of $\mathbf{V}(x)$ along the vector fields of $\mathbf{F}[f]$ is given  by
\begin{align}
    \mathcal{L}_{\mathbf{F}[f]}\mathbf{V}(x)=\{(x-x_d)^\top\xi, \xi\in\mathbf{F}[f](x)\}.
\end{align}
On the other hand, and in view of \eqref{eq:F}, one has for all $\xi\in\mathbf{F}[f](x)$
\begin{align}
    (x-x_d)^\top\xi=\begin{cases}
    -k\|x-x_d\|^2,\\
    \qquad x\in\mathbf{int}(\mathcal{X})\;\mathrm{or}\;\nu(x)^\top\kappa_0(x)\leq 0,\\
    -k\alpha\|x-x_d\|^2-k(1-\alpha)(x-x_d)^\top\\
    \hfill\Pi(\nu(x))(x-x_d),\\
    \qquad x\in\partial\mathcal{X}\;\mathrm{and}\;\nu(x)^\top\kappa_0(x)\geq 0.
    \end{cases}
\end{align}
Now, since the matrix $\alpha I_n+(1-\alpha)\Pi(\nu(x))$ is positive semi-definite (remember that $\Pi(\nu(x))$ is positive semi-definite), one has
\begin{align}
    \max \mathcal{L}_{\mathbf{F}[f]}\mathbf{V}(x)\leq 0, \quad \forall x\in\mathcal{X}.
\end{align}
If follows from \cite[Theorem 1]{Cortes2008DiscontinuousSystems} that $x=x_d$ is a strongly stable equilibrium of \eqref{eq:inclusion}. Moreover, since $x_d\in\mathbf{int}(\mathcal{X})$ there exists $r>0$ such that $\mathcal{B}(x_d,r)\subset\mathbf{int}(\mathcal{X})$ and hence, on this set, the dynamics reduces to $\dot x=-k(x-x_d)$ which shows exponential stability of $x=x_d$. Finally, the points where $0\in\mathcal{L}_{\mathbf{F}[f]}\mathbf{V}(x)$ correspond to either $x=x_d$ or $\nu(x)^\top(x-x_d)< 0$ and $(x-x_d)^\top\Pi(\nu(x))(x-x_d)=0$. The latter implies that $\Pi(\nu(x))(x-x_d)=0$ or equivalently $x-x_d=\lambda\nu(x)$ with $\lambda=\nu(x)^\top(x-x_d)<0$. By \cite[Theorem 2]{Cortes2008DiscontinuousSystems}, the proof is complete. 
\end{proof}

Theorem \ref{theorem:convergence} shows that all solutions will be safe and also converge to either $x_d$ or to the set of points $\mathcal{E}$ on the boundary. These are equilibrium points from which convergence to the target is not possible. Note that, since $\partial\mathcal{X}$ is Lebesgue measure zero  in $\mathbb{R}^n$, the subset $\mathcal{E}\subset\partial\mathcal{X}$ is also of Lebesgue measure zero.  It remains to study (depending on the given worskpace $\mathcal{X}$) the invariance properties of these equilibrium points (stable, unstable) and whether they are isolated or not. 
However, studying the invariance properties of the undesired equilibria  heavily depends on the considered free space $\mathcal{X}$ (e.g., topological properties). For example, if the set $\mathcal{X}$ is not path-connected, there might be initial conditions where the robot can never converge to the target $x_d$ regardless of the control law applied. In this case, the robot will converge inevitably to another equilibrium point in the set $\mathcal{E}$. Next, we concentrate our attention to the topologically simple environments; known as sphere worlds.
\subsection{Euclidean Sphere Worlds}
Let us consider a sphere world as defined in \cite{Koditschek1990RobotBoundary}. In other words, we assume that $\mathcal{X}$ consists of one large ball $\mathcal{O}_0:=\overline{\mathcal{B}(c_0,r_0)}, c_0:=0,$ which bounds the workspace minus $M$ smaller disjoint  balls $\mathcal{O}_i:=\mathcal{B}(c_i,r_i), i=1\cdots M,$ that define obstacles in $\mathbb{R}^n$ that are strictly contained in the interior of the workspace, i.e.,
\begin{align}\label{eq:sphere-world}
    \mathcal{X}:=\mathcal{O}_0\setminus\bigcup_{i=1}^M\mathcal{O}_i.
\end{align}
\begin{assumption}\label{assumption:disjoint}
Obstacles are separated from each other, i.e.,
\begin{align}
    \|c_i-c_j\|>r_i+r_j,\quad \forall i,j\in\{1,\cdots,M\}, i\neq j,
\end{align}
and from the boundary, i.e.,
\begin{align}
    \|c_i\|+r_i>r_0,\quad \forall i\in\{1,\cdots,M\}.
\end{align}
\end{assumption}
Since the obstacles are disjoint, the boundary of $\mathcal{X}$ is identified by the union of $(M+1)$ spheres
\begin{align}
\partial\mathcal{X}=\bigcup_{i=0}^M\partial\mathcal{O}_i=\bigcup_{i=0}^M\{x\in\mathbb{R}^n: \|x-c_i\|^2=r_i^2\}.
\end{align}
For each $x\in\{x\in\mathbb{R}^n: \|x-c_i\|^2=r_i^2\}$, the normal to the $i-$th sphere $\partial\mathcal{O}_i$ is given by
\begin{align}
\nu_0(x)&=+\frac{(x-c_0)}{\|x-c_0\|},\\
    \nu_i(x)&=-\frac{(x-c_i)}{\|x-c_i\|}, \quad i=1,\cdots,M.
\end{align}
Now we can characterize explicitly the set $\mathcal{E}$ of undesirable equilibria and its stability properties. 
\begin{theorem}[Sphere Worlds]
Consider a free space that is described by the sphere worlds $\mathcal{X}\subset\mathbb{R}^n$ as in \eqref{eq:sphere-world} under Assumption \ref{assumption:disjoint}. Consider the kinematic system \eqref{eq:dx1} under the control law \eqref{eq:kappa:projection}  with $\kappa_0(\cdot)$ as in \eqref{eq:kappa0}. Then,
\begin{enumerate}
\item The unique Filippov solution converges to the set $\{x_d\}\cup_{i=1}^M\{\bar x_i\}$ with $\bar x_i:=(1-\alpha_i)x_d+\alpha_i c_i$ and $\alpha_i:=1+r_i\|x_d-c_i\|^{-1}$. 
\item Each undesired equilibrium point $x=\bar x_i$ is unstable.
\item The desired equilibrium $x=x_d$ is locally exponentially stable and almost globally asymptotically stable.
\end{enumerate}
\end{theorem}
\begin{proof}
First, we show that the set $\mathcal{E}$ defined in Theorem \ref{theorem:convergence} contains exactly $M$ isolated points.  This set can be written as follows:
\begin{align}
    \mathcal{E}=\bigcup_{i=0}^M\mathcal{E}_i:=\bigcup_{i=0}^M\{x\in\partial\mathcal{O}_i: x=x_d+\lambda_i\nu_i(x), \lambda_i\in\mathbb{R}_{<0}\}.
\end{align}
For $i\in\{1,\cdots,M\}$, let $x\in\mathcal{E}_i$ then $x-x_d=\lambda_i\nu_i(x)$ and $\lambda_i=(x-x_d)^\top\nu_i(x)=-(x-x_d)^\top(x-c_i)/r_i<0$. Also, we have $\|x-x_d\|^2=\lambda_i(x-x_d)^\top\nu_i(x)=\lambda_i^2$. It follows that
\begin{align*}
    \|c_i-x_d\|^2&=\|c_i-x\|^2+\|x-x_d\|^2-2(x-c_i)^\top(x-x_d)\\
    &=r_i^2+\lambda_i^2+2\lambda_ir_i=(r_i+\lambda_i)^2.
\end{align*}
Hence, $\lambda_i=\pm\|c_i-x_d\|-r_i$. However, $\lambda_i=\|c_i-x_d\|-r_i$ is ruled out by the fact that $x_d\in\mathbf{int}(\mathcal{X})$ ($\|c_i-x_d\|\geq r_i$) and $\lambda$ must be negative. Thus $\lambda_i=-\|c_i-x_d\|-r_i$ and plugging this back into $x-x_d=\lambda_i\nu_i(x)=-\lambda_i(x-c_i)/r_i$ we get $(1+\lambda_i/r_i)x=x_d+(\lambda_i/r_i)c_i$ or equivalently $x=(1-\alpha_i)x_d+\alpha_i c_i$. On the other hand, for $i=0$, we get $\lambda_0=\pm\|c_0-x_d\|+r_0$ which is non-negative thanks to $x_d\in\mathbf{int}(\mathcal{X})$ ($\|x_d\|\leq r_0$) and hence $\mathcal{E}_0=\emptyset$.

To prove the instability of the isolated equilibrium $\bar x_i:=(1-\alpha_i)x_d+\alpha_i c_i$, it is sufficient to find an open ball $\{x:\|x-\bar x_i\|<\epsilon\}$ such that, for some $x_0$ arbitrary close to $\bar x_i$, the solution $x(t)$ that starts at $x_0$  must leave this ball. Let the closed set $\mathcal{P}_i:=\{x:\nu_i(x)^\top(x-x_d)\leq 0\}$ and define the  compact set $\mathcal{S}_i:=\partial\mathcal{O}_i\cap\mathcal{P}_i$.  When restricted on $\mathcal{S}_i$, $x$ evolves according to the dynamics
\begin{align}\label{eq:dx-reduced}
    \dot x= -k\Pi(\nu_i(x))(x-x_d).
\end{align}
Consider the positive define function with respect to $\bar x_i$
\begin{align}
    \mathbf{W}_i(x)=\frac{1}{2}\|x-\bar x_i\|^2.
\end{align}
The time derivative of $\mathbf{W}_i(x)$ for all $x\in\mathcal{S}_i$ satisfies
\begin{align}
   \nonumber\dot{\mathbf{W}}_i(x)&=-k(x-\bar x_i)\Pi(\nu_i(x))(x-x_d)\\ \nonumber&=-k(1-\alpha_i)(x-x_d)^\top\Pi(\nu_i(x))(x-x_d)\\
   \nonumber&\qquad-\alpha(x-c_i)^\top\Pi(\nu_i(x))(x-x_d)\\
   \nonumber&=-k(1-\alpha_i)(x-x_d)^\top\Pi(\nu_i(x))(x-x_d)\geq0
\end{align}
where we have used the fact that $(x-c_i)^\top\Pi(\nu_i(x))=0$ to obtain the third equality and $\Pi(\nu_i(x))$ is positive semi-definite to obtain the last inequality. Furthermore, $\dot{\mathbf{W}}_i(x)=0$ implies the existence of $\lambda_i$ such that $x=x_d+\lambda_i\nu_i(x)$. But $x\in\mathcal{S}_i$ implies that $\lambda<0$; which shows that $x\in\mathcal{E}_i=\{\bar x_i\}$. On the other hand, since $(\bar x_i-c_i)^\top(\bar x_i-x_d)=\alpha(\alpha
-1)\|c_i-x_d\|^2>0$, one has $\bar x_i\in\mathbf{int}(\mathcal{P}_i)$ and hence there exists an open ball $\{x:\|x-\bar x_i\|<\epsilon\}\subset\mathcal{P}_i$. Pick an initial condition $x_0$ in $\partial\mathcal{O}_i\cap\{x:\|x-\bar x_i\|<\epsilon\}\subset\mathcal{S}_i$ that is arbitrary close to $\bar x_i$. The solution $x(t)$ that starts at $x_0$ cannot approach $\bar x_i$ while in $\mathcal{S}_i$ since $\mathbf{W}_i(x)>0$ and $\dot{\mathbf{W}}_i(x)>0$ for all $x\in\mathcal{S}_i\setminus\{\bar x_i\}$. However, since no other equilibria is inside $\mathcal{S}_i$ except $\bar x_i$, the solution must leave the set $\partial\mathcal{O}_i\cap\{x:\|x-\bar x_i\|<\epsilon\}$ in finite time. However, it cannot leave the set through $\partial\mathcal{O}_i$ since  $\|x-c_i\|^2$ is constant along \eqref{eq:dx-reduced}. Consequently, the solution must leave the ball $\{x:\|x-\bar x_i\|<\epsilon\}$ in finite time and the equilibrium $x=\bar x_i$ is unstable. Finally, since the set of undesired equilibria $\cup_{i=1}^M\{\bar x_i\}$ is unstable and has Lebesgue measure zero in $\mathbb{R}^n$, it follows that $x=x_d$ is attractive from almost all initial conditions and thus it is almost globally asymptotically stable. 
\end{proof}
 
Geometrically speaking, the unstable equilibrium $\bar x_i$ is nothing but the antipodal point (on the boundary of the obstacle) which is diametrically opposite to $x_d$. Note that, for sphere worlds, our piecewise continuous vector field guarantees safety and convergence from almost all initial conditions without requiring unknown parameter tuning compared to the navigation functions approach \cite{Koditschek1990RobotBoundary}. The undesired equilibria can be removed by considering the hybrid approach in \cite{Berkane2019}. 

\begin{figure}
    \centering
    \includegraphics[width=0.48\columnwidth]{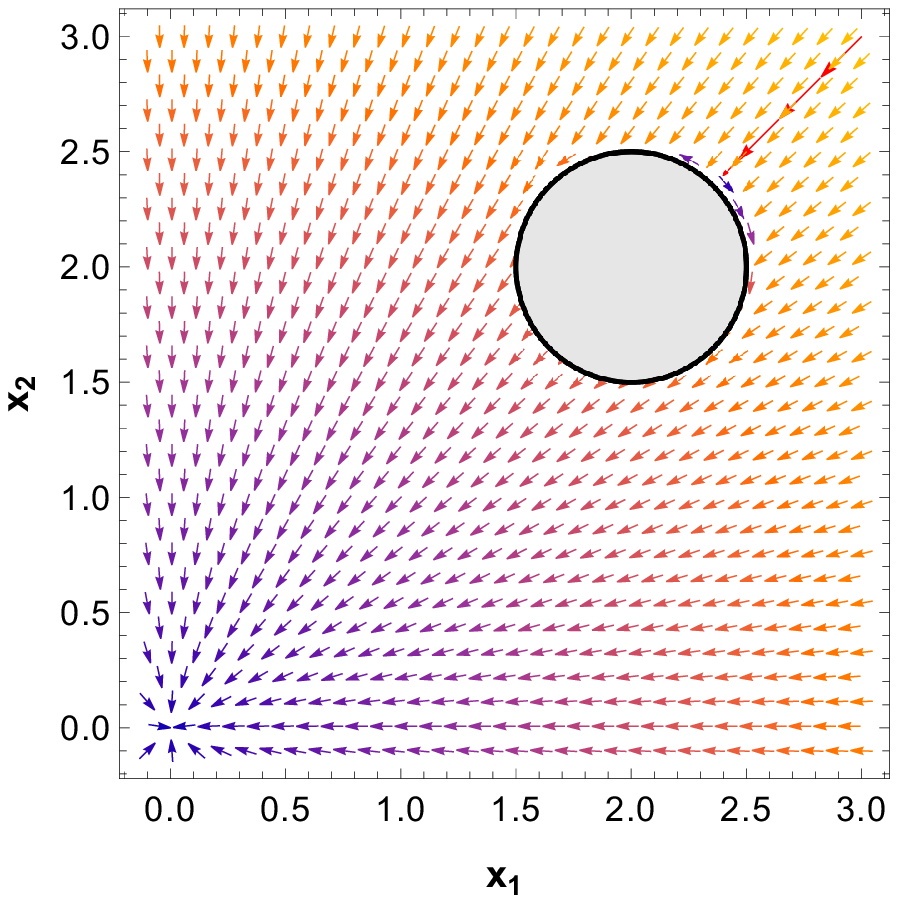}
    \includegraphics[width=0.48\columnwidth]{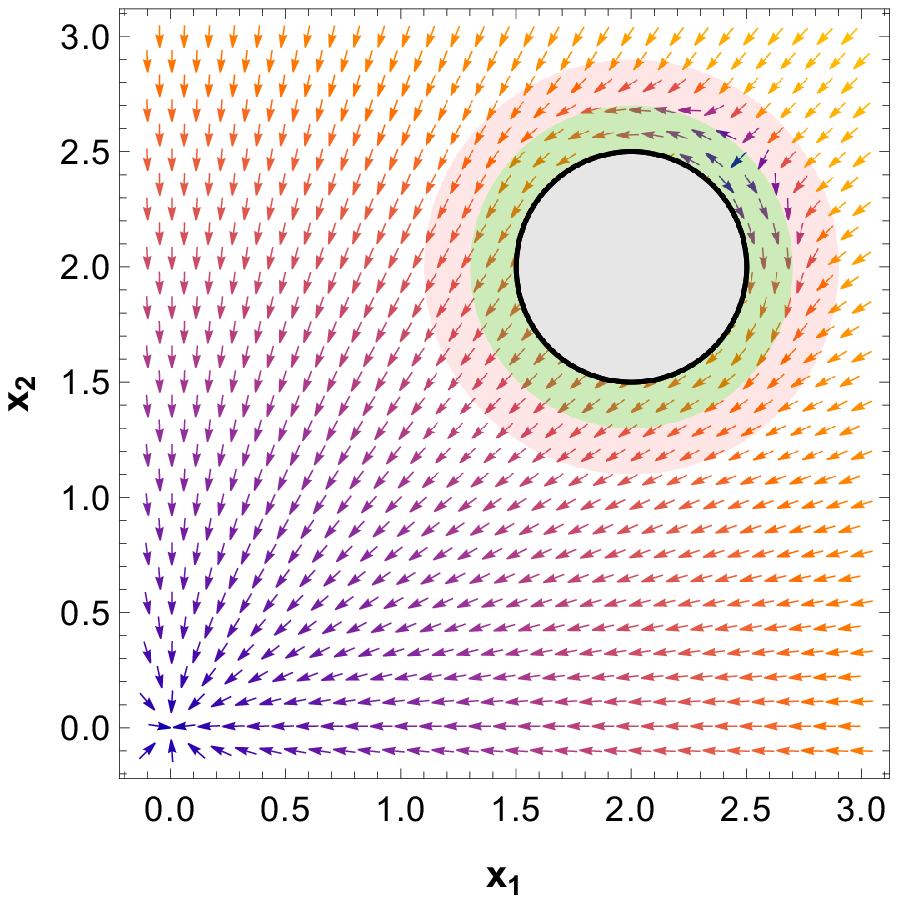}
    \caption{Left: vector fields generated by the discontinuous controller \eqref{eq:kappa:projection}, when $\mathcal{X}=\complement\mathcal{B}([2,2],0.5)$ and $x_d=0$. The convergence to $x_d$ is ensured from all initial conditions except from the stable manifold (half-line in red) of the unstable equilibrium $\bar x_1=(1+1/4\sqrt{2})[2,2]$ (antipodal point). The intensity of the vector field decreases from orange to blue colors. Right: vector fields generated by the continuous controller \eqref{eq:u:practical}. The continuous implementation allows to maintain a safety margin of size $\epsilon=0.2$ (in green) from the obstacle while the controller exhibits a smooth transition from the stabilization mode to the avoidance mode once it enters the set $\mathcal{X}_\epsilon\setminus\mathcal{X}_{\epsilon^\prime}$ (in red). }
    \label{fig:my_label}
\end{figure}

\section{Distance-based characterization and Continuous vector fields generation}\label{section:distance}
Here, we assume that the obstacle region $\complement\mathcal{X}$ has {\it positive reach}. In other words, there exists $h>0$ such that, at each $x$ with $\mathbf{d}_{\complement\mathcal{X}}(x)<h$, the projection $\mathbf{P}_{\partial\mathcal{X}}(x)$ is unique \cite[Thm 6.3, Chap. 6]{Delfour2011ShapesGeometrics}. For instance, the set in Fig.~\ref{fig:nagumo} is not positive reach since points that are arbitrary closed to the corners of the square (boundary of the workspace) can have two nearest points to the boundary. On the other hand,  the set in Fig.~\ref{fig:my_label} is positive reach. 

In practice, it is preferable to keep the robot away from the boundary by ensuring a minimum safety margin. For this purpose, we pick a safety margin $\epsilon>0$ and define
\begin{align}
    \mathcal{X}_\epsilon:=\{x\in\mathbb{R}^n:\mathbf{d}_{\complement\mathcal{X}}(x)\geq\epsilon\}
\end{align}
as our (practical) free space. Note that $\nabla\mathbf{d}_{\complement\mathcal{X}}(x)$ represents the inward unit normal at the boundary of $\mathcal{X}_\epsilon$. Therefore, our outward normal vector $\nu(x)$ at $\partial\mathcal{X}_\epsilon$ is nothing but $\nu(x)=-\nabla\mathbf{d}_{\complement\mathcal{X}}(x)$. Since $\complement\mathcal{X}$ has positive reach, it follows from \cite[Thm 3.3, Chap. 6]{Delfour2011ShapesGeometrics} that, for all $x$ with $s\leq\mathbf{d}_{\complement\mathcal{X}}(x)\leq h$, $0<s<h$, $\nabla\mathbf{d}_{\complement\mathcal{X}}(x)$ is Lipstichtz continuous and satisfies
\begin{align}\label{eq:nabla-d}
    \nabla\mathbf{d}_{\complement\mathcal{X}}(x)=\frac{x-\mathbf{P}_{\partial\mathcal{X}}(x)}{\|x-\mathbf{P}_{\partial\mathcal{X}}(x)\|}=\frac{x-\mathbf{P}_{\partial\mathcal{X}}(x)}{\mathbf{d}_{\complement\mathcal{X}}(x)}.
\end{align}
Therefore, the discontinuous controller $u=\kappa_{\mathcal{X}_\epsilon}(x)$ can be applied and the results of Section \eqref{section:main-smooth} holds for the set $\mathcal{X}_\epsilon$. However, the discontinuity in the vector field is undesirable in practice. We propose the following continuous version that is inspired from the smoothing mechanism in \cite{Praly1991AdaptiveSystems}:
\begin{align}\label{eq:u:practical}
    \hat\kappa_{\mathcal{X}_\epsilon}(x)=\begin{cases}
    \kappa_0(x)&\mathbf{d}_{\complement\mathcal{X}}(x)>\epsilon^\prime\;\mathrm{or}\;\\
    &\kappa_0(x)^\top\nabla\mathbf{d}_{\complement\mathcal{X}}(x)\geq 0,\\
    \hat\Pi(x)\kappa_0(x)&\mathbf{d}_{\complement\mathcal{X}}(x)\leq\epsilon^\prime\;\mathrm{and}\;\\
    &\kappa_0(x)^\top\nabla\mathbf{d}_{\complement\mathcal{X}}(x)\leq 0,
    \end{cases}
\end{align}
where $0<\epsilon<\epsilon^\prime\leq h$ and 
\begin{align}
    \hat\Pi(x)&:=I_n-\phi(x)\nabla\mathbf{d}_{\complement\mathcal{X}}(x)\nabla\mathbf{d}_{\complement\mathcal{X}}(x)^\top,\\
    \phi(x)&:=\min\left(1,\frac{\epsilon^\prime-\mathbf{d}_{\complement\mathcal{X}}(x)}{\epsilon^\prime-\epsilon}\right).
\end{align}
The control law above is continuous since for $\mathbf{d}_{\complement\mathcal{X}}(x)=\epsilon^\prime$, one has $\phi(x)=0$
 and, thus, $\hat\kappa_{\mathcal{X}_\epsilon}(x)=\kappa_0(x)$. 
\begin{figure}[h!]
    \centering
    \includegraphics[width=0.9\columnwidth]{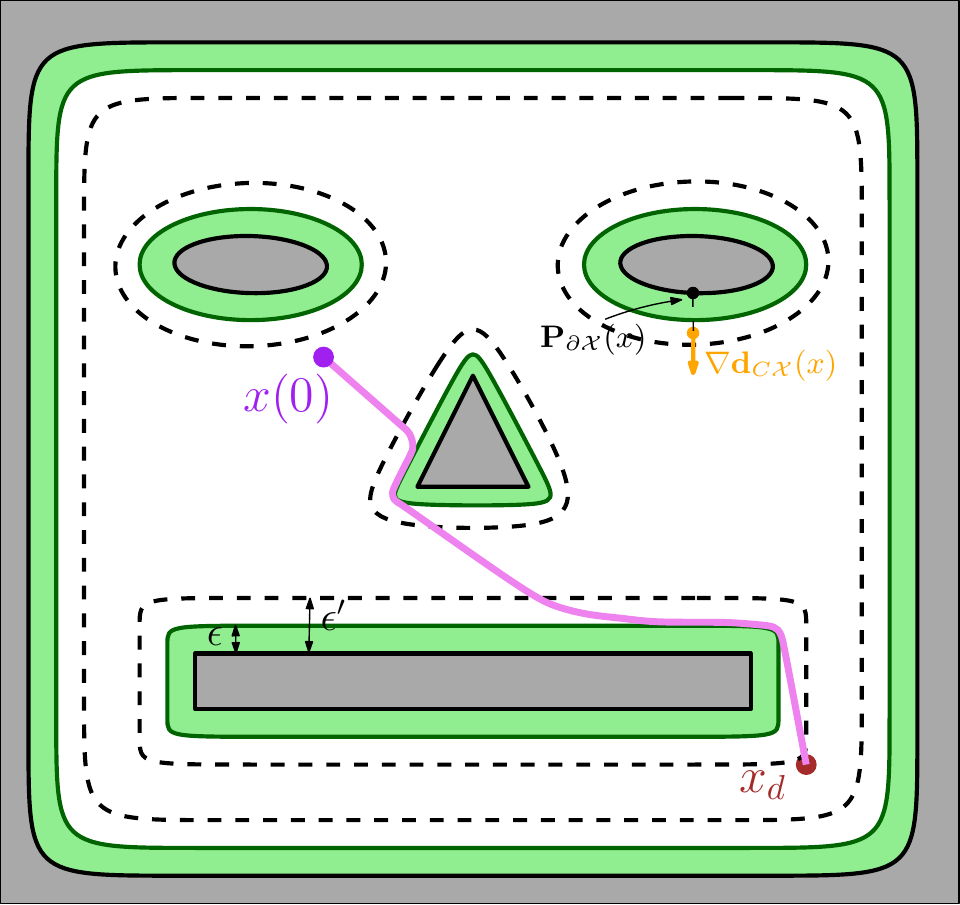}
    \caption{The continuous controller \eqref{eq:u:practical} depends explicitly on the distance to the obstacle set $\complement\mathcal{X}$ (dark gray) and the bearing vector $\nabla\mathbf{d}_{\complement\mathcal{X}}(x)$; both can be obtained from on-board sensor measurements such as range scanners. The resulting smooth navigation trajectories (purple) maintains a safety margin  of size $\epsilon$ (green) from the obstacles. The avoidance maneuver starts at a distance of $\epsilon^\prime>\epsilon$.}
    \label{fig:my_label}
\end{figure}

\section{Sensor-based implementation}\label{section:practical}
The advantage of the continuous control implementation in \eqref{eq:u:practical} is its characterization using the distance function $\mathbf{d}_{\complement\mathcal{X}}(x)$ and the projection $\mathbf{P}_{\partial\mathcal{X}}(x)$. Fortunately, both $\mathbf{d}_{\complement\mathcal{X}}(x)$ and $\mathbf{P}_{\partial\mathcal{X}}(x)$ can be obtained locally in  real-time using simple on-board sensors such as LiDARs or stereo cameras. In this section, we provide some insights about the practical implementation of our proposed navigation strategy using 2D and 3D sensing models.
\subsection{2D navigation using a LiDAR range scanner}
In a 2D setting, we assume available a LiDAR (light detection and ranging) range scanner  with an angular scanning range of $360^o$ and a fixed radial range of $R>0$. At each position $x\in\mathcal{X}$, the sensory measurement of the LiDAR scanner can be modelled by a polar curve $\rho(\cdot;x):[-\pi,\pi)\to[0,R]$, which defined as
\begin{align}\label{eq:rho}
   \rho(\theta;x):=\min\left(R,\min_{\substack{y\in\partial\mathcal{X}\\\atantwo(y-x)=\theta}}\|y-x\|\right). 
\end{align}
Therefore, it is possible to obtain the distance to the boundary when the latter is less or equal to $R$ from
\begin{align}\label{eq:b:2D}
    \mathbf{d}_{\complement\mathcal{X}}(x)=\min_{\theta}\rho(\theta;x).
\end{align}
Hence, we can implement the navigation controller \eqref{eq:u:practical} by selecting further restricting $\epsilon^\prime< R$. Moreover, the bearing vector $\nabla\mathbf{d}_{\complement\mathcal{X}}(x)$ can be obtained from
\begin{align}\label{eq:v:2D}
    \nabla\mathbf{d}_{\complement\mathcal{X}}(x)=-(\cos(\theta^*),\sin(\theta^*)),\quad\theta^*=\mathrm{arg}\min_{\theta}\rho(\theta;x).
\end{align}
Consequently, the navigation controller \eqref{eq:u:practical} has been fully characterized using a LiDAR scanner measurement model. 
\begin{figure}
    \centering
    \includegraphics[width=0.45\columnwidth]{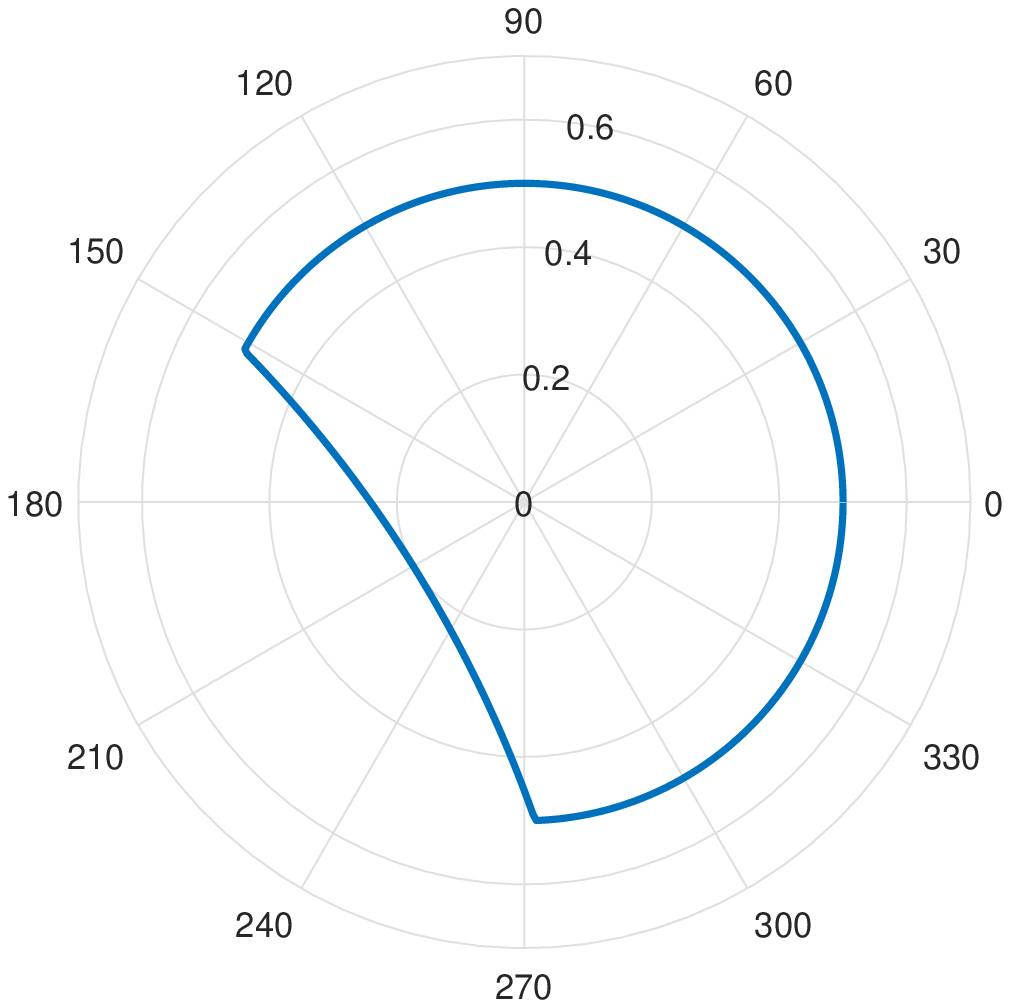}
    \includegraphics[width=0.45\columnwidth]{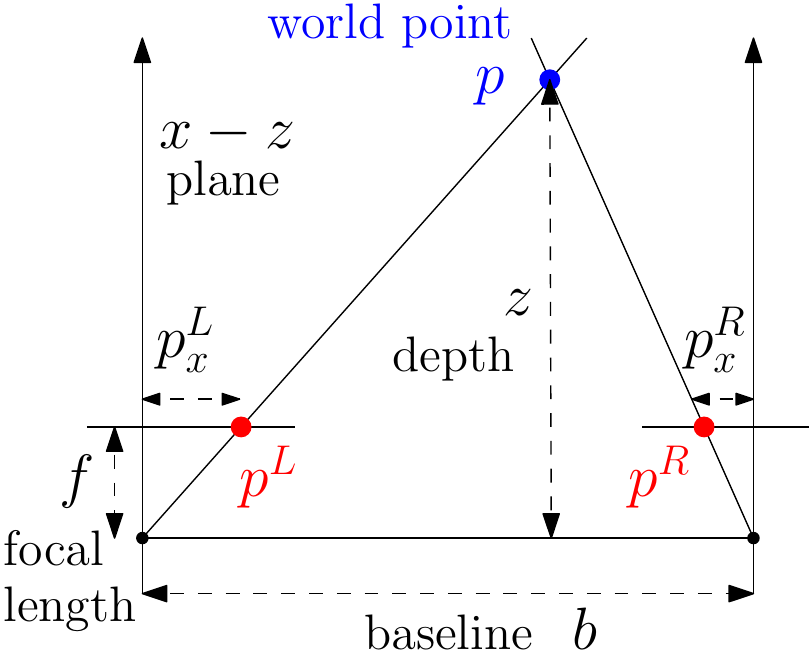}
    \caption{Left: example of 2D LiDAR polar curve in \eqref{eq:rho} that allows to extract the minimum distance and bearing vector for navigation. Right: computation of depth (3D range) from the image points of a stereo camera.}
    \label{fig:sensors}
\end{figure}
\subsection{3D navigation using a depth map from a stereo camera}
In a 3D scenario, a stereo camera allows to extract a local depth map that can be used to implement our navigation controller. For stereo cameras with parallel optical axes, focal length $f$, baseline $b$, corresponding image points $p^L:=(p_x^L,p_y^L)$ and $p^R:=(p_x^R,p_y^R)$ in the left and right cameras respectively, the depth can be calculated from (see Fig.~\ref{fig:sensors})
\begin{align}
    z=\frac{f\cdot b}{|p_x^L-p_x^R|}.
\end{align}
Since the cameras are assumed parallel and calibrated, we have $p_y^L=p_y^R$. The distance $\mathbf{d}_{\complement\mathcal{X}}(x)$ to the boundary can be hence derived by minimizing the depth $z$ across all pixels $(p_x^L,p_y^L)$ of the camera's image (the darkest pixel on the depth map). Let $(\bar p_x^L,\bar p_y^L)$ be an argument of this minimization.  If we assume that the center of the frame is taken at the left camera, the 3D Cartesian coordinates $p:=(p_x,p_y,p_z)$ of the closest boundary point $p$  expressed in the frame that is generated by the front image plane are
\begin{align}
    p_{x(y)}=\frac{b\cdot\bar p_{x(y)}^L}{(\bar p_x^L-\bar p_x^R)}, \quad p_z=\frac{b\cdot f}{(\bar p_x^L-\bar p_x^R)}.
\end{align}
If we let $(R_L,p_L)$ be the homogeneous coordinates of the left camera with respect to robot's center of gravity at $x$, the bearing vector from $x$ to the closest point $p$ on the boundary is given by
$
   \nabla\mathbf{d}_{\complement\mathcal{X}}(x)=(R_Lp+p_L)/\|R_Lp+p_L\|.
$
Consequently, the navigation controller \eqref{eq:u:practical} has been fully characterized using a stereo camera sensing model.
\section{Numerical simulations}\label{section:simulation}
We consider a free space $\mathcal{X}\subseteq\mathbb{R}^2$ such that:
\begin{multline}
    \mathcal{X}:=\{x=(x_1,x_2)\in\mathbb{R}^2:     f_i(x)\leq 0, i\in\{1,2\},\\
   \mathrm{and}\;g_q(x)\leq 0, q\in\mathbb{Z}
    \},
\end{multline}
where $f_1(x)=4\sin(x_1)-x_2-5, f_2(x)=4\sin(x_1)+x_2-5,$ and $g_q(x)=4-(x_1-(4q+3)\pi/2)^2-x_2^2$. The free space $\mathcal{X}$ consists of the closed (but unbounded) area between two sinusoidal-shaped boundaries minus an infinite number of spherical obstacles centered at $((4q+3)\pi/2,0), q\in\mathbb{Z},$ with radius equals $2$. To simulate the function of the range scanner (at $1^o$ resolution), we consider finding the solutions $(\lambda_\theta^1,\lambda_\theta^2,\beta_\theta^1,\beta_\theta^2,\cdots)$ of the single-variable nonlinear equations
$f_1(x+\lambda_\theta^1 v_\theta)=0, f_2(x+\lambda_\theta^2 v_\theta)=0,$ and $g_q(x+\beta_\theta^q v_\theta)=0$, for each position $x$ and direction $v_\theta=(\cos(\theta),\sin(\theta)), \theta=0^o,1^o,\cdots,359^o$. Note that the interval of search for the solution is limited to $[0,R]$ where $R$ is the radial range. If no solution is found in this range then we take the corresponding scalar as $R$. The value of the polar curve at this position and angle is given by $\rho(\theta;x)=\min\{\lambda_\theta^1,\lambda_\theta^2,\beta_\theta^1,\cdots\}$. The distance to the boundary and the gradient vector are computed according to \eqref{eq:b:2D} and \eqref{eq:v:2D}, respectively. For simulation, we pick the desired reference at $x_d=(-9,3)$, the controller parameters as $k=0.5, \epsilon=0.2, \epsilon^\prime=0.4$, and the sensory radius as $R=0.5$. We consider different initial conditions for our robot. Fig.~\ref{fig:simulation} shows the obtained (continuously differentiable) safe navigation trajectories.
\begin{figure}
    \centering
    \includegraphics[width=\columnwidth]{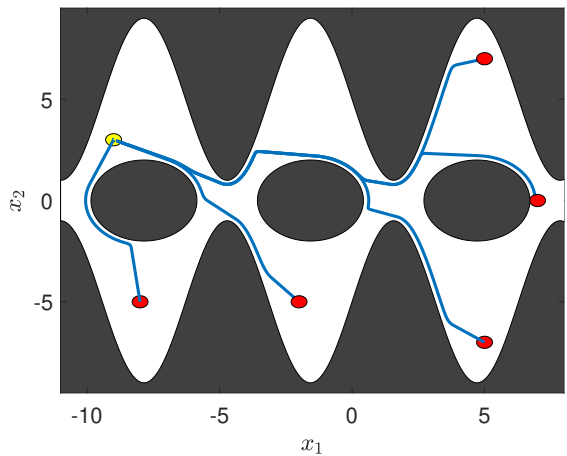}
    \caption{The obtained (continuously differentiable and piecewise smooth) navigation trajectories of the proposed sensor-based control law in \eqref{eq:u:practical}, starting at a set of initial positions (red), converge to the goal (yellow) while avoiding the obstacles region $\mathbb{R}^2\setminus\mathcal{X}$ (dark gray).  A simulation video is provided at \url{https://youtu.be/37ImfGoPyqg}.}
    \label{fig:simulation}
\end{figure}
\section{Conclusion}\label{section:conclusion}
We proposed a computationally simple sensor-based navigation controller that allows for a robot to safely move in an $n-$dimensional unknown environment. The controller switches between two modes of navigation: stabilization and avoidance. The stabilization mode uses a nominal control law and is activated away from the boundary while the avoidance mode projects the nominal control law onto the SVCs to ensure a minimally invasive safe controller. For metrically simple free spaces (sphere worlds), the navigation is guaranteed from almost everywhere.  Moreover, we proposed a smoothing mechanism that removes the discontinuity in the controller by characterizing our free space using the distance to the obstacles function. This characterization allows also to implement our proposed navigation strategy using on-board sensors without any prior knowledge of the environment. As a future work, we aim to consider more complex robot dynamics in the design of the control law.
\bibliographystyle{IEEEtran}
\bibliography{references.bib}
\end{document}